\documentclass[12pt]{amsart}
\usepackage{amssymb}
\usepackage{epsfig}
\usepackage{graphicx}
\date{29 May 2008}

\title{A non-crossing standard monomial theory}

\author{T. Kyle Petersen}\address{Department of Mathematics, University of Michigan,
  Ann Arbor MI 48109--1043}
\email{tkpeters@umich.edu}

\author{Pavlo Pylyavskyy}\address{Department of Mathematics, University of Michigan,
  Ann Arbor MI 48109--1043}
\email{pylyavskyy@gmail.com}

\author{David E Speyer}\address{Department of Mathematics, MIT,
  Cambridge, MA 02139-4307}
\email{speyer@math.mit.edu}
\thanks{Third author supported by a Research Fellowship from the Clay Mathematics Institute.}

\newlength{\cellsz}
\newcounter{cellsize}
\newcommand{\setcellsize}[1]{%
  \setcounter{cellsize}{#1}%
  \setlength{\cellsz}{\value{cellsize}\unitlength}}%
\setcellsize{10}%
\newcommand\cellify[1]{\def\thearg{#1}\def\nothing{}%
\ifx\thearg\nothing \vrule width0pt height\cellsz depth0pt\else
\hbox to 0pt{{\begin{picture}(\value{cellsize},\value{cellsize})
  \put(0,0){\line(1,0){\value{cellsize}}}
  \put(0,0){\line(0,1){\value{cellsize}}}
  \put(\value{cellsize},0){\line(0,1){\value{cellsize}}}
  \put(0,\value{cellsize}){\line(1,0){\value{cellsize}}} \end{picture} \hss}}\fi%
\vbox to \cellsz{ \vss \hbox to \cellsz{\hss$#1$\hss} \vss}}
\newcommand\tableau[1]{\vcenter{\vbox{\let\\\cr
\baselineskip -16000pt \lineskiplimit 16000pt \lineskip 0pt
\ialign{&\cellify{##}\cr#1\crcr}}}}
\newcommand\tabl[1]{\vtop{\let\\\cr
\baselineskip -16000pt \lineskiplimit 16000pt \lineskip 0pt
\ialign{&\cellify{##}\cr#1\crcr}}}

\theoremstyle{plain}
\newtheorem{theorem}{Theorem}[section]
\newtheorem{question}{Question}
\newtheorem{proposition}[theorem]{Proposition}
\newtheorem{lemma}[theorem]{Lemma}
\newtheorem{corollary}[theorem]{Corollary}
\newtheorem{conjecture}[theorem]{Conjecture}

\theoremstyle{definition}
\newtheorem{definition}[theorem]{Definition}
\newtheorem{example}[theorem]{Example}

\theoremstyle{remark}
\newtheorem{remark}[theorem]{Remark}

\def\from{\leftarrow}
\def\Hull{\mathrm{Hull}}

\def\X{\,\,\lower2pt\hbox{\input{figX.pstex_t}}}
\def\noXv{\,\,\lower2pt\hbox{\input{figNoX.pstex_t}}}
\def\noXh{\,\,\lower2pt\hbox{\input{figNoX2.pstex_t}}}
\def\noXDU{\,\,\lower2pt\hbox{\input{figDU.pstex_t}}}
\def\noXDD{\,\,\lower2pt\hbox{\input{figDD.pstex_t}}}
\def\noXUD{\,\,\lower2pt\hbox{\input{figUD.pstex_t}}}
\def\noXUU{\,\,\lower2pt\hbox{\input{figUU.pstex_t}}}
\def\noXRR{\,\,\lower2pt\hbox{\input{figRR.pstex_t}}}
\def\noXRL{\,\,\lower2pt\hbox{\input{figRL.pstex_t}}}
\def\noXLR{\,\,\lower2pt\hbox{\input{figLR.pstex_t}}}
\def\noXLL{\,\,\lower2pt\hbox{\input{figLL.pstex_t}}}
\def\XRR{\,\,\lower2pt\hbox{\input{figXRR.pstex_t}}}
\def\XRL{\,\,\lower2pt\hbox{\input{figXRL.pstex_t}}}
\def\XLR{\,\,\lower2pt\hbox{\input{figXLR.pstex_t}}}
\def\XLL{\,\,\lower2pt\hbox{\input{figXLL.pstex_t}}}

\def\Z{\mathbb{Z}}
\def\R{\mathbb{R}}
\def\ZZ{\Z}
\def\RR{\R}
\def\CC{\mathbb{C}}

\DeclareMathOperator{\mnnL}{{\bf m}^{\mathit{(nn)}}_{\mathit L}}
\DeclareMathOperator{\mncL}{{\bf m}^{\mathit{(nc)}}_{\mathit L}}
\DeclareMathOperator{\mwsL}{{\bf m}^{\mathit{(ws)}}_{\mathit L}}

\begin{document}

\begin{abstract}
The second author has introduced non-crossing tableaux, objects whose non-nesting analogues are semi-standard Young tableaux. We relate non-crossing tableaux to Gelfand-Tsetlin patterns and develop the non-crossing analogue of standard monomial theory. Leclerc and Zelevinsky's weakly separated sets are special cases of non-crossing tableaux, and we suggest that non-crossing tableaux may help illuminate the theory of weakly separated sets.
\end{abstract}

\maketitle

\section{Introduction}

Consider an $m \times m$ matrix $X$ of formal variables $x_{ij}$. For $I$ a subset of $[m] := \{ 1,2, \ldots, m \}$, let $p_I$ be the determinant of the submatrix of $X$ whose columns are indexed by $I$ and whose rows are the bottom $|I|$ rows. The $p_I$ are known as \emph{Pl\"ucker coordinates}. The subalgebra of $\CC[x_{ij}]$ generated by the Pl\"ucker coordinates is called the \emph{Pl\"ucker algebra}. Obviously, the Pl\"ucker algebra is spanned, as a $\CC$-vector space, by monomials in the $p_I$. One of the classical topics in combinatorial commutative algebra is \emph{standard monomial theory}. It is the construction of a collection of monomials, indexed by semi-standard Young tableaux, which give a $\CC$-basis for the Pl\"ucker algebra. Specifically, $\prod_{r=1}^N p_{I_r}$ is a standard monomial if and only if the sets $I_1$, $I_2$, \dots, $I_N$ are the columns of a semi-standard Young tableaux. The relations expressing other monomials in terms of the standard monomials are classically known as \emph{straightening relations}.

In~\cite{P}, the second author showed that semi-standard Young tableaux could, in a certain sense, be considered as non-nesting objects. There is a general philosophy that for every non-nesting object there is a corresponding non-crossing object. The paper \cite{P}~introduced a notion of non-crossing tableaux, and showed that there is a bijection between non-nesting tableaux (i.e., semi-standard Young tableaux) and non-crossing tableaux, preserving shape and content. In this paper, we will explain how to build a basis of the Pl\"ucker algebra using non-crossing tableaux (implied by the slightly more general Corollary \ref{cor:cr}). As a corollary, we obtain a new degeneration of the flag variety to a Stanley-Reisner complex (Corollary \ref{thm:nctri}). 

It is not clear yet what the best use of this technology will be; our view is that non-crossing tableaux are a tool that should be tried whenever semi-standard Young tableaux are applicable but not productive. One clue that non-crossing tableaux should be useful is that the non-crossing condition is a relaxation of the weak separation condition of~\cite{LZ}. (The semi-standard condition is neither stronger nor weaker than weak separation.) Weakly separated sets give commutative subrings of the quantum deformation of the Pl\"ucker algebra, and, conjecturally, are clusters in the cluster structure on the Pl\"ucker algebra. There are not enough collections of weakly separated sets to give a basis of the Pl\"ucker algebra (they give linearly independent, but not spanning, monomials). In the theory of cluster algebras, this discrepancy is (conjecturally) fixed by adding additional generators to the Pl\"ucker algebra. Non-crossing tableaux offer a compromise, giving a basis that coincides with the cluster construction when possible, but still uses only monomials in the Pl\"ucker coordinates. In particular, in the case of the Grassmannian $G(2,m)$, (that is, when we limit ourselves to $2 \times 2$ minors of $X$), the notions of non-crossing tableaux and weakly separated sets coincide. The corresponding degeneration of the Grassmannian has components indexed by the set of triangulations of an $m$-gon. This is an occurrence of the standard example of a correspondence between non-nesting objects and non-crossing objects---that the number of standard Young tableaux of shape $2 \times (m-2)$ is the same as the number of triangulations of an $m$-gon.

We now sketch our approach. Following ideas of~\cite{MS}, we first describe a sagbi basis of the Pl\"ucker algebra. From the perspective of algebraic geometry, this corresponds to giving a degeneration of the flag variety to the toric variety corresponding to the Gelfand-Tsetlin polytope. This tells us that one can construct bases of the Pl\"ucker algebra indexed by Gelfand-Tsetlin patterns, but does not construct any particular such basis. In the standard monomial theory, one then triangulates the Gelfand-Tsetlin polytope, in a triangulation whose faces are indexed by semi-standard Young tableaux. We give a different triangulation of the Gelfand-Tsetlin polytope, whose faces are indexed by non-crossing tableaux, and thus obtain a new set of standard monomials. To do this, we use the usual correspondence between semi-standard Young tableaux and collections of non-crossing lattice paths; we then replace the non-crossing tableaux by ``non-kissing'' lattice paths. We warn the reader that \emph{non-nesting} tableaux correspond to \emph{non-crossing} paths while \emph{non-crossing} tableaux correspond to \emph{non-kissing} paths.

The paper is organized as follows. Section \ref{sec:compat} presents the \emph{compatibility conditions} of non-nesting, non-crossing, and weak separation, as well as the concept of a \emph{compatibility complex}. Section \ref{sec:GT} defines the cone of Gelfand-Tsetlin patterns and reinterprets compatibility conditions in terms of lattice paths. Section \ref{sec:cr} relates the cone of Gelfand-Tsetlin patterns to the Pl\"ucker algebra. In particular, this section contains Corollary \ref{cor:cr}, which describes how to form new linear bases of the Pl\"ucker algebra. Section \ref{sec:driving} presents \emph{driving rules}, a general idea for triangulating the cone of Gelfand-Tsetlin patterns, from which both non-nesting and non-crossing triangulations are then derived. The consequences of driving rules for non-nesting tableaux are described in Section \ref{sec:Michigan}; and the consequences for non-crossing tableaux in Section \ref{sec:Boston}. Section \ref{sec:regular} explores further properties of the non-crossing triangulation.

\clearpage

\section{Compatibility complexes}\label{sec:compat}

Let $2^{\bf m}$ denote the set of all subsets of $[m]$.

Call two pairs $\{u<v\}$ and $\{x<y\}$, with $u$, $v$, $x$ and $y$ in $\ZZ \cup \{ \infty \}$, \emph{non-nesting} if \textbf{neither} of the following two conditions holds:
\begin{enumerate}
\item $x<u<v<y$,
\item $u<x<y<v$.
\end{enumerate}
Let $I = \{i_1 < i_2 < \cdots < i_k\}$ and $J = \{j_1 < j_2 < \cdots < j_l\}$ be two elements of $2^{\bf m}$, with $k \leq l$. Now we remove the ``common part" of $I$ and $J$: if for any $s$ we have $i_s = j_s$, remove $i_s$ and $j_s$ from $I$ and $J$. After several such steps we have formed $I' = \{i'_1 < \cdots < i'_{k'}\}$ and $J' =\{ j'_1 < \cdots < j'_{l'}\}$ so that $i'_s \neq j'_s$ for any $s$. 

\begin{definition}[Non-nesting rule]
Subsets $I$ and $J$ are non-nesting if:
\begin{enumerate}
  \item for all $1 \leq s \leq k'-1$, pairs $\{i'_s, i'_{s+1}\}$ and $\{j'_s, j'_{s+1}\}$ are non-nesting;
  \item if $k' < l'$ then $\{i'_{k'}, \infty \}$ and $\{ j'_{k'}, j'_{k'+1} \}$ are non-nesting.
\end{enumerate}
\end{definition}

\begin{example}
The sets $\{1,3,4\}$ and $\{2,3,5\}$ do not nest, while $\{1,3,5\}$ and $\{2,3,4\}$ do.
\end{example}

We will later encounter a simple graphical interpretation for the non-nesting rule.

\begin{remark}
The non-nesting condition is exactly the condition satisfied by columns of semi-standard Young tableaux; see \cite{P} for details.
\end{remark}

Let ${\bf m}^{(nn)}$ denote the set of subsets of $2^{\bf m}$ whose elements are pairwise non-nesting. We can think of this set as a simplicial complex.

\begin{example}
If $m =2$ we have four vertices: $\emptyset, \{1\}, \{2\}, \{1,2\}$, all pairwise non-nesting. Thus, every subset of these four vertices is a face, and ${\bf 2}^{(nn)}$ is nothing but a $3$-simplex.
\end{example}

We will examine the complex ${\bf m}^{(nn)}$ in a slightly more general setting in Section \ref{sec:Michigan}, as well as some analogous complexes. The non-nesting condition is the first example of what we call a \emph{compatibility condition}. In general, a compatibility condition, $c$, is any rule by which we select a collection of unordered pairs of subsets, $(I,J) = (J,I)$, with $I, J \in 2^{\bf m}$. The corresponding \emph{compatibility complex} ${\bf m}^{(c)}$ is a simplicial complex whose vertex set is $2^{\bf m}$, with faces given by subsets of $2^{\bf m}$ whose elements are pairwise compatible with respect to $c$.

We now introduce another compatibility condition. Call two pairs $\{u < v\}$ and $\{x<y\}$ \emph{non-crossing} if \textbf{neither} of the following two conditions holds:
\begin{enumerate}
\item $x<u<y<v$,
\item $u<x<v<y$.
\end{enumerate}
As before, let $I =\{ i_1 < i_2 < \cdots < i_k\}$ and $J = \{j_1 < j_2 < \cdots < j_l\}$ be two elements of $2^{\bf m}$, with $k \leq l$, and remove their common part to form $I'$ and $J'$.

\begin{definition}[Non-crossing rule]
Subsets $I$ and $J$ are non-crossing if:
\begin{enumerate}
  \item for all $1 \leq s \leq k'-1$, pairs $\{i'_s, i'_{s+1}\}$ and $\{j'_s, j'_{s+1}\}$ are non-crossing,
 \item  if $k' < l'$ then $\{i'_{k'}, \infty \}$ and $\{ j'_{k'}, j'_{k'+1} \}$ are non-crossing.
\end{enumerate}
\end{definition}

\begin{example}
Both $(\{1,3,5\},\{2,3,4\})$ and $(\{1,4,5\},\{2,3,6\})$ form non-crossing pairs, while sets $\{1,3,4\}$ and $\{2,3,5\}$ cross.
\end{example}

As with non-nesting, we will later present a simple graphical interpretation for this rule. Let ${\bf m}^{(nc)}$ denote the compatibility complex for the non-crossing rule.

\begin{remark}
The non-crossing condition is directly related to \emph{non-crossing tableaux} defined in \cite{P}. Note however that the ``semi-standard" version of the non-crossing rule adopted in \cite{P} is different from the rule given here. We prefer the current rule because it respects the isomorphism between the Grassmannians $G(k,m)$ and $G(m-k,m)$, and because it is more closely related to the condition of weak separation in~\cite{LZ}.
\end{remark}


We now present a compatibility condition first defined by Leclerc and Zelevinsky \cite{LZ} and later studied in \cite{Sc1}, \cite{Sc2}.  Write $I - J$ for the set $\{i \mid i \in I, i \not \in J\}$, and $I \prec J$ if $i < j$ for all $i\in I$, $j\in J$.

\begin{definition}[Weak separation rule]
Subsets $I$ and $J$ are \emph{weakly separated} if at least one of the following holds:
\begin{enumerate}
  \item $|I| \geq |J|$ and $J-I$ can be partitioned into a disjoint union $J-I = J' \sqcup J''$ so that $J' \prec I-J \prec J''$,
  \item $|I| \leq |J|$ and $I-J$ can be partitioned into a disjoint union $I-J = I' \sqcup I''$ so that $I' \prec J-I < I''$.
\end{enumerate}
\end{definition}

\begin{example}
The sets $\{ 2,3,4\}$ and $\{1,3,5\}$ are weakly separated and non-crossing, whereas $\{1,4,5\}$ and $\{2,3,6\}$ are non-crossing but they are not weakly separated. The sets $\{1,3,4\}$ and $\{2,3,5\}$ are neither non-crossing nor weakly separated.
\end{example}

The preceding example suggests the following lemma, which follows immediately given the graphical interpretations for non-crossing and weak separation given in Theorem \ref{thm:path}. Let ${\bf m}^{(ws)}$ be the compatibility complex for the weak separation rule.

\begin{lemma} \label{lem:ncws}
Weak separation is a strengthening of the non-crossing rule. In other words, if $(I,J)$ is an edge of ${\bf m}^{(ws)}$, then $(I,J)$ is an edge of ${\bf m}^{(nc)}$.
\end{lemma}

Later we will see that ${\bf m}^{(nc)}$ is pure with predictable dimension, and although ${\bf m}^{(ws)}$ is a proper subcomplex, it is (conjecturally) pure of the same dimension.

For any compatibility condition $c$, we can study subcomplexes of the compatibility complex induced by restricting the vertex set in some way. For example, we can restrict the size of allowable subsets: ${\bf m}_{k\leq l}^{(c)}$ denotes the complex whose vertices are those subsets of $[m]$ with at least $k$ and no more than $l$ elements. 
We will present a more general kind of restriction of the vertex set in Section \ref{sec:GT}.

As examples of the types of questions we will try to answer, we present a theorem and conjecture for weakly separated compatibility complexes.

\begin{theorem} \cite[Theorem 1.3]{LZ} \label{thm:lzkl}
The dimension of ${\bf m}^{(ws)}_{k \leq l}$ is $$\binom{m+1}{2} - \binom{m+1-l}{2} - \binom{k+1}{2} + 1.$$
\end{theorem}



\begin{conjecture} \cite[Conjecture 1]{Sc1} \label{conj:lzk}
Complex ${\bf m}^{(ws)}_{k \leq l}$ is pure.
\end{conjecture}

Though we are unable to use our approach to prove Conjecture~\ref{conj:lzk}, we are able to prove analogous results for the non-nesting and non-crossing compatibility complexes; one consequence is that we obtain Theorem~\ref{thm:lzkl} as a corollary.

\begin{remark}
There are many other compatibility rules. For example there is the \emph{strong separation} rule of \cite{LZ} and the \emph{semi-noncrossing} rule of \cite{P}. One can also conjure up a rule which would relate to the non-nesting rule the same way weak separation relates to non-crossing. Later in the paper we consider the whole family of rules which we call \emph{driving rules}. Among the rules mentioned above the non-nesting, the non-crossing and the semi-noncrossing rules fall into this category.
\end{remark}

\section{Gelfand-Tsetlin patterns}\label{sec:GT}

Let $L_m$ be the diagram of the staircase shape $(m-1, \ldots, 1)$ in French notation. Call the southwest corner of the shape {\it {the source}}, and call the northeast corners (including two degenerate nodes) {\it {the sinks}}. More generally we can form a \emph{partial staircase} $L \subset L_m$ by selecting a subset of the sinks of $L_m$ and all the nodes to the south and west of these sinks. See Figure  \ref{fig:dr3}, where the sinks have coordinates $(1,5)$, $(3,3)$, $(4,2)$, and $(6,0)$.  We refer to the squares of the square grid as \emph{cells}. We say that a cell is a \emph{cell of $L$} if either its northwest or its southeast corner is a node of $L$; we say that a cell is an \emph{internal cell} if all of its nodes are in $L$ and \emph{external} otherwise. We say that two external cells are \emph{connected} to each other if they are horizontally or vertically adjacent and the edge separating them does not lie in $L$. In Figure~\ref{fig:dr3}, the edges of $L$ are drawn with solid lines, while each connected class of external cells is surrounded by a dotted line.\footnote{The conventions on external cells are technical, and the reader may first want to consider just the cases where $L$ is a rectangle or a staircase.} We divide the cells into \emph{connectivity classes}, where two cells are in the same connectivity class if and only if there is a chain of cells linking one to the other in which each consecutive pair is connected to each other.


For a partial staircase $L$, a \emph{Gelfand-Tsetlin pattern} $T$ of shape $L$ is a filling of the cells of $L$ with non-negative integers so that entries weakly increase along rows and columns, any two connected external cells contain the same entry, and the external cells to the west of $L$ contain zeroes. A real-valued Gelfand-Tsetlin pattern of shape $L$ is a filling of the cells of $L$ with non-negative real numbers obeying the same condition. An example of a Gelfand-Tsetlin pattern of partial staircase shape is given in Figure \ref{fig:dr3}.

\begin{figure}
\begin{center}
\includegraphics{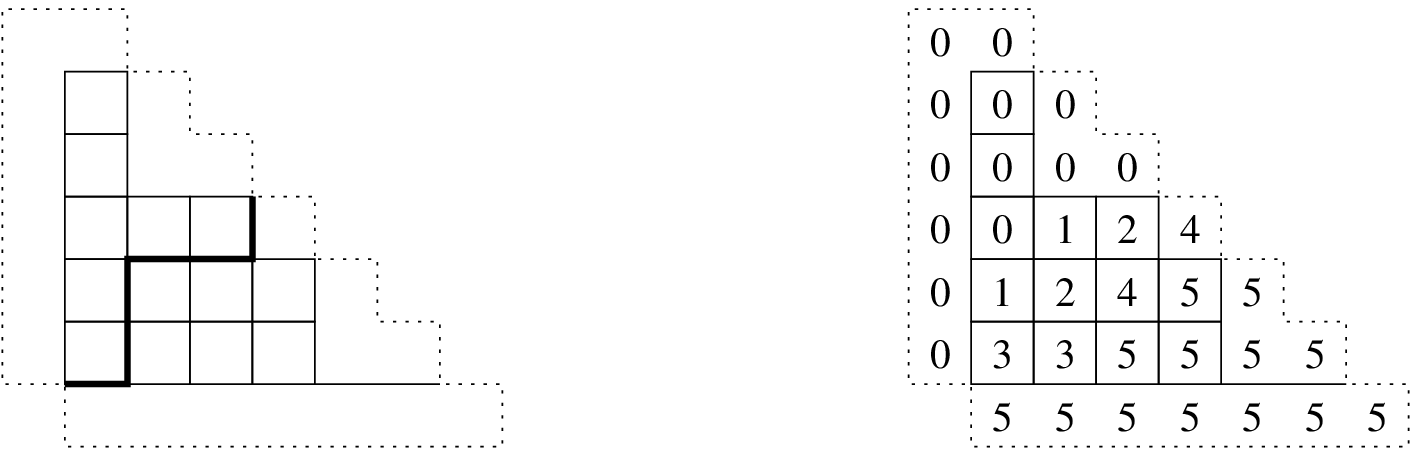}
\end{center}
\caption{}\label{fig:dr3}
\end{figure}

To each path along the edges of $L_m$ from the source to a sink moving in north and east directions we can associate a word of length $m$ on the alphabet $\{N,E\}$. This word is easily transformed into an element $I \in 2^{\bf m}$ by taking the index of north steps along the path, so that, for example, $ENNEEN \leftrightarrow \{2,3,6\}$ (see the left side of Figure \ref{fig:dr3}). Clearly this is a bijection between the set of paths on $L_m$ and $2^{\bf m}$. We denote by $p=p(I)$ the path corresponding to $I$. 

Note that $p$ might or might not fit into a partial staircase $L$. Thus, the shape of $L$ can be viewed as a restricting condition on the subsets of $2^{\bf m}$. We denote by ${\bf m}_L^{(c)}$ the compatibility complex with vertex set $I \in 2^{\bf m}$ such that $p(I)$ fits inside $L$. For example, if $L$ is given by selecting sinks $(m-l,l), (m-l+1,l-1), \ldots, (m-k,k)$, we have ${\bf m}^{(c)}_L = {\bf m}_{k\leq l}^{(c)}$.

We will now see the compatibility rules we encountered have a nice description in terms of corresponding paths. 


\begin{definition}[Non-crossing paths]
Paths $p$ and $q$ are \emph{non-crossing} if $p$ always lies to the northwest of $q$, or always lies to the southeast of $q$. 
\end{definition}

Let us call a \emph{common part} of $p$ and $q$ a maximal (by inclusion) connected subpath (possibly a single node) which is shared by both of them. (The ``common part" of subsets $I$ and $J$ corresponds the vertical edges in the common parts of paths $p(I)$ and $p(J)$.) A common part is called \emph{proper} if it does not include the source or any sink. 

\begin{definition}[Non-kissing paths]
Paths $p$ and $q$ are \emph{non-kissing} if $p$ and $q$ leave every proper common part in same direction by which they entered.
\end{definition}

Note that every proper common part of a non-crossing pair is a kiss, and every proper common part of a non-kissing pair is a crossing.

Figure \ref{fig:dr4} illustrates the above definitions. Two pairs of paths are shown locally. Suppose in both cases the paths share one proper common part consisting of steps $N$, $E$, $N$, $N$, $E$. The pair on the left is kissing but non-crossing, while the pair on the right is crossing but non-kissing. 

\begin{figure}
\begin{center}
\input{dr4.pstex_t}
\end{center}
\caption{}\label{fig:dr4}
\end{figure}

We introduce one more piece of terminology. Let $p$ and $q$ be a pair of paths. We define a pair of nodes $v=(v_1,v_2)$ in $p$ and $w=(w_1,w_2)$ in $q$, to be an \emph{hourglass} if $v$ is due northwest of $w$ (i.e., $v_1=w_1-k$ and $v_2=w_2+k$ for some $k \geq 0$), the edges of $p$ leading into and out of $v$ are respectively pointed $E$ and $N$ and the edge of $q$ leading into and out of $w$ are respectively pointed $N$ and $E$. (We also define nodes $v$ and $w$ to be an hourglass if the same holds with the roles of $p$ and $q$ interchanged.) 

 
Now we are ready to state the following theorem.

\begin{theorem} \label{thm:path}
Let $I$ and $J$ be subsets of $[m]$ with corresponding paths $p=p(I)$ and $q=p(J)$. Then,
\begin{enumerate}
  \item $I$ and $J$ are non-nesting if and only if $p$ and $q$ are non-crossing,
  \item $I$ and $J$ are non-crossing if and only if $p$ and $q$ are non-kissing,
  \item $I$ and $J$ are weakly separated if and only if $p$ and $q$ are non-kissing, cross at most once and do not have any hourglasses. 
\end{enumerate}
\end{theorem}

\begin{proof}
By the definitions of the compatibility conditions, the common parts of $p$ and $q$ do not influence anything and each can be reduced to a single point. When the common parts are just single nodes the needed equivalences are easily verified.
\end{proof}

Note that Theorem \ref{thm:path} implies Lemma \ref{lem:ncws}.

\subsection{The cone of GT-patterns}

Recall our definition of real-valued Gelfand-Tsetlin patterns above. The real-valued Gelfand-Tsetlin patterns of shape $L$ form a polyhedral cone in $\RR^{\mathrm{Cells}(L)}$ which we denote $GT_L$. To each path $p$ in $L$ let us associate a GT-pattern $T_p \in GT_L$ which has a $1$ in those cells which lie to the southeast of $p$ and has a zero in those cells which lie to the northwest . Note that $T_p = T_q$ if and only if $p = q$ (this is guaranteed by the presence of the external cells and fails if we consider only internal cells).

It is easy to see that $GT_L$ is the positive span of the vectors $T_p$, as $p$ ranges over paths contained in $L$. For a given element $T$ of $GT_L$, there can be many ways to decompose $T$ into a linear combination of $T_p$. Figure \ref{fig:dr5} shows two different ways to decompose the GT-pattern from Figure \ref{fig:dr3}. Note that in the first way any two participating paths are non-crossing, while in the second way any two participating paths are non-kissing. In the light of Theorem \ref{thm:path} this suggests the following definition and question.

\begin{figure}[h!]
\begin{center}
\input{drf3.pstex_t}
\end{center}
\caption{}\label{fig:dr5}
\end{figure}

Let us say that complex ${\bf m}^{(c)}_L$ \emph{triangulates} $GT_L$ if every element of $GT_L$ can be uniquely decomposed into a positive linear combination of $T_p$, where the participating $p$ form a face in ${\bf m}^{(c)}_L$. In other words, the cone $GT_L$ can be written as the union of interiors of simplicial cones, where the rays of the simplicial cones are the sets $\{ T_p \}_{p \in P}$, with $P$ running over the sets in ${\bf m}^{(c)}_L$. 

\begin{question} \label{q:1}
Which compatibility conditions yield complexes that triangulate $GT_L$?
\end{question}
%

We also introduce the notation $GT^0_L$ for the convex hull of the points $T_p$, as $p$ runs over paths contained in $L$. The polytope $GT^0_L$ is a slice of the $GT_L$ by an affine hyperplane. Namely, let $c$ be the cell whose northwest corner is at the origin and whose southeast corner is at $(1,-1)$. For every path $p$, $T_p^c=1$, so $GT^0_L$ is the polytope of Gelfand-Tsetlin patterns that have a $1$ in cell $c$. So, if ${\bf m}^{(c)}_L$ triangulates $GT_L$, then $GT^0_L$ is the union of the interiors of the simplices whose vertices lie at the $c$-compatible sets. In other words, ${\bf m}^{(c)}_L$ encodes the combinatorics of a triangulation of the polytope $GT^0_L$. This allows us to deduce topological properties of ${\bf m}^{(c)}_L$. Namely, if ${\bf m}^{(c)}_L$ triangulates $GT_L$ then ${\bf m}^{(c)}_L$ is homeomorphic to a ball of the same dimension as $GT^0_L$.

This raises the question of the dimension of $GT^0_L$. Let $C$ be the number of connectivity classes of cells of $L$. In $GT^0_L$, the cell whose northwest corner is the origin always contains a $1$ and the cell whose southeast corner is the origin always contains a $0$. So the dimension of $GT^0_L$ is at most $C-2$ and a little more thought shows that this is, in fact, the dimension. We set $N(L)=C-2$. We define a \emph{critical node} of $L$ to be a node which has edges leaving it heading towards the north and east; it is easy to check that $N(L)$ is also the number of critical nodes. Summing up, we have the following.

\begin{proposition}  \label{pro:trigeom}
 If  ${\bf m}^{(c)}_L$ triangulates $GT_L$ then the simplicial complex ${\bf m}^{(c)}_L$ is homeomorphic to a ball of dimension $N(L)$. 
\end{proposition}

\begin{remark}\label{rem:RW}
This remark concerns the relation between this work and that of Reiner and Welker~\cite{RW}. The polytope $GT^0_L$ is the \emph{order polytope} of a certain obvious poset, whose elements are the critical nodes of $L$. The non-nesting triangulation of $GT^0_L$, which we will describe in Section~\ref{sec:Michigan}, is the canonical triangulation of the order polytope. The non-crossing triangulation, described in Section~\ref{sec:Boston}, resembles the equatorial triangulation of the order polytope, in that both are multicones on spheres. However, the non-crossing triangulation is not the same as the equatorial triangulation. As triangulations of $GT^0_L$, they are not equal even when $L$ is the $2 \times 2$ rectangle. As abstract simplicial complexes, the first case we know of where the equatorial and non-crossing complexes are nonisomorphic is when $L$ is the $3 \times 4$ rectangle. In the equatorial complex, there are edges joining each pair of points in the triple $(136, 147, 257)$, but the triangle with these three vertices is not a face of the equatorial complex. In the non-crossing complex, by contrast, the minimal non-faces are always edges. (This example is based on Example~3.22 of~\cite{RW}.)
\end{remark}

\section{Coordinate rings}\label{sec:cr}

For background on material used in this section we refer the reader to \cite[Chapter 14]{MS}. Note that we consider the numbering of rows to be bottom-justified, in order to be consistent with our way of positioning GT-patterns.

Let $X = \{x_{ij}\}$ be an $m$ by $m$ matrix of indeterminates. To each $I \in 2^{\bf m}$ we associate a minor $x_I$ of $X$ with column set $I$ and row set $1,\ldots,|I|$. In \cite{MS}, the algebra generated by all such minors is called the \emph{Pl\"ucker algebra}. The complete flag variety $\mathfrak Fl_m$ is the multi-proj of the Pl\"ucker algebra. The minors under consideration are called the \emph{Pl\"ucker coordinates} of the variety. Similarly, if we take the algebra generated only by minors of certain sizes ${\bf a} = (a_1, \ldots, a_l)$, its multi-proj is the \emph{partial flag variety} $\mathfrak Fl_m^{\bf a}$.

Let $R$ be a subalgebra of a polynomial ring $\CC[x_1,x_2, \ldots, x_N]$. Recall that a subset 
$\{ \phi_1, \ldots, \phi_r \}$ of $R$ is called a \emph{sagbi basis} of $R$ with respect to a given term order if every monomial which occurs as an initial term of an element of $R$ occurs also as an initial term of a monomial in $\{ \phi_1, \ldots, \phi_r \}$. The following theorem appeared in \cite{KM} and is also stated in \cite{MS}. Although it is stated there for complete flags, the proof extends verbatim to the partial case. A \emph{diagonal} or \emph{antidiagonal} term order is any order such that, in any minor, the the diagonal or the antidiagonal monomial (respectively) is the initial term.

%

\begin{theorem} \cite[Theorem 14.11]{MS} \label{thm:cr}
The Pl\"ucker coordinates (bottom justified minors) form a sagbi basis for the Plucker algebra under any diagonal or antidiagonal term order.
\end{theorem}

The reader may also want to consult~\cite[Chapter~11]{Sturm}, where sagbi bases are referred to as canonical bases.

The \emph{diagonal semigroup} $\mathfrak D_m^{\bf a}$ is generated by exponent matrices of diagonal terms of Pl\"ucker coordinates of $\mathfrak Fl_m^{\bf a}$. For example, if $m=6$ and ${\bf a} = (2,3,5)$, then some of the generators of   $\mathfrak D_m^{\bf a}$ can be seen below.

\[ 
\underbrace{\tableau{{}&{}&{}&{}&{}&{} \\ {}&{}&{}&{}&{1}&{} \\ {}&{}&{}&{1}&{}&{} \\ {}&{}&{1}&{}&{}&{} \\ {}&{1}&{}&{}&{}&{} \\ {1}&{}&{}&{}&{}&{}},  
\tableau{{}&{}&{}&{}&{}&{} \\ {}&{}&{}&{}&{}&{1} \\ {}&{}&{}&{1}&{}&{} \\ {}&{}&{1}&{}&{}&{} \\ {}&{1}&{}&{}&{}&{} \\ {1}&{}&{}&{}&{}&{}}, \ldots}_{\binom{6}{5}},
\underbrace{\tableau{{}&{}&{}&{}&{}&{} \\ {}&{}&{}&{}&{}&{} \\ {}&{}&{}&{}&{}&{} \\ {}&{}&{1}&{}&{}&{} \\ {}&{1}&{}&{}&{}&{} \\ {1}&{}&{}&{}&{}&{}}, 
\tableau{{}&{}&{}&{}&{}&{} \\ {}&{}&{}&{}&{}&{} \\ {}&{}&{}&{}&{}&{} \\ {}&{}&{}&{1}&{}&{} \\ {}&{1}&{}&{}&{}&{} \\ {1}&{}&{}&{}&{}&{}},\ldots}_{\binom{6}{3}},
\underbrace{\tableau{{}&{}&{}&{}&{}&{} \\ {}&{}&{}&{}&{}&{} \\ {}&{}&{}&{}&{}&{} \\ {}&{}&{}&{}&{}&{} \\ {}&{1}&{}&{}&{}&{} \\ {1}&{}&{}&{}&{}&{}}, \tableau{{}&{}&{}&{}&{}&{} \\ {}&{}&{}&{}&{}&{} \\ {}&{}&{}&{}&{}&{} \\ {}&{}&{}&{}&{}&{} \\ {}&{}&{1}&{}&{}&{} \\ {1}&{}&{}&{}&{}&{}},\ldots}_{\binom{6}{2}}.
\]

\begin{theorem} \label{thm:semi}
Semigroups $\mathfrak D_m^{\bf a}$ and $GT_L$ are isomorphic, where $L \subset L_m$ is a partial staircase shape with sinks $(m-a_1,a_1), \ldots, (m-a_l,a_l)$. 
\end{theorem}

\begin{proof}
We proceed as in the proof of \cite[Theorem 14.23]{MS}, so that it suffices to exhibit an automorphism of a group $\mathbb Z^{N}$ containing $\mathfrak D_m^{\bf a}$ and $GT_L$ which induces a bijection between their generating sets. If $d_{i,j}$ is the entry in a cell $(i,j)$ of an element of $\mathfrak D_m^{\bf a}$, and $g_{i,j}$ is the entry in a cell $(i,j)$ of an element of  $GT_L$, then the desired automorphism is given by $g_{i,j} = \sum_{j \leq i' \leq i+j-1} d_{i',j}$. It is not hard to check that this is indeed an automorphism, and that it provides the needed bijection of generators: $x_I \mapsto T_{p(I)}$.

For example, the GT-pattern in Figure \ref{fig:dr3} is the image of the following element of $\mathfrak D_6^{(2,3,5)}$:
\[ 
\tableau{{}&{}&{}&{}&{}&{} \\ {}&{}&{}&{}&{}&{} \\ {}&{}&{}&{}&{}&{} \\ {}&{}&{}&{1}&{1}&{2} \\ {}&{1}&{1}&{2}&{1}&{} \\ {3}&{}&{2}&{}&{}&{}}.
\]
\end{proof}

The complicated rules regarding external cells and connectivity were set up exactly to make Theorem~\ref{thm:semi} hold.

We define a \emph{${\bf m}^{(c)}_L$-tableau  supported on face $F$} to be a collection of vertices of ${\bf m}^{(c)}_L$ that belong to $F$, such that each element of $F$ occurs at least once. The terminology is motivated by the following observation: a $\mnnL$-tableau is precisely the columns of a semistandard tableau whose entries and column lengths are restricted by $L$. For example, the first decomposition of a GT-pattern on Figure \ref{fig:dr5} is non-crossing, which means that the sets involved form a $\mnnL$-tableau. When arranged into a Young shape the sets form the columns of the following semistandard tableau:
\setcellsize{13}
\[ 
\tableau{{1}&{1}&{1}&{3}&{3} \\ {2}&{3}&{4}&{4}&{5} \\ {4}&{5}&{6}&{6}}.
\]
The shape $L$ in this case induces the restriction that columns must have $2$, $3$, or $5$ rows and that entries are no greater than $6$.

A different choice of compatibility condition yields  different kinds of tableaux. For example, the following tableau corresponds to the second decomposition of the GT-pattern in Figure \ref{fig:dr5}, and is an example of a $\mncL$-tableau.
\[ 
\tableau{{1}&{1}&{3}&{3}&{1} \\ {3}&{5}&{4}&{4}&{2} \\ {4}&{6}&{5}&{6}}.
\]
Note that, in the non-crossing case, there is no obvious order in which to arrange the columns. Our standard practice is to order them lexicographically.

To each ${\bf m}^{(c)}_L$-tableaux $T$ one can associate a monomial in Pl\"ucker variables by taking product of Pl\"ucker coordinates $x_T = \prod_{I \in T} x_I$ over all columns $I$ of $T$. We call such $x_T$ the \emph{${\bf m}^{(c)}_L$-monomials}. The following corollary of Theorem \ref{thm:cr} provides motivation for Question \ref{q:1}.

\begin{corollary} \label{cor:cr}
If ${\bf m}^{(c)}_L$ triangulates the cone $GT_L$, the ${\bf m}^{(c)}_L$-monomials form a linear basis for the corresponding subalgebra of the Pl\"ucker algebra.
\end{corollary}

\begin{remark}
In the particular case of $\mnnL$ one just recovers the ``standard monomial basis", labeled by the semi-standard tableaux.
\end{remark}

This correspondence allows us to deduce that all triangulating compatibility complexes share a certain combinatorial invariant. For a $d$-dimensional simplicial complex, let $f_i$ denote the number of $i$-dimensional faces, and collect these numbers in the \emph{$f$-vector} $(1, f_0,\ldots,f_d)$.

\begin{corollary}\label{cor:fvec}
If ${\bf m}^{(c)}_L$ triangulates the cone $GT_L$, the Poincar\'e series of the corresponding subring of the Pl\"ucker algebra is given by \[\sum_{i} f_i({\bf m}^{(c)}_L) \sum_{j \geq i+1} \binom{j-1}{i} t^j.\] All complexes ${\bf m}^{(c)}_L$ triangulating $GT_L$ have the same $f$-vector.
\end{corollary}

\begin{proof}
We first prove the formula for the Poincar\'e series. 

Grade tableaux by the number of columns, and let $r_{i,j}$ denote the number of degree $j$ tableaux $T$ supported on an $i$-dimensional face $F$. Clearly then the Poincar\'e series is given by:
\[ \sum_{i} f_i({\bf m}^{(c)}_L) \sum_{j \geq i+1} r_{i,j} t^j.\]
To obtain the desired formula notice that $r_{i,j}$ simply counts the ways to split a collection of $j$ columns into $i+1$ nonempty subsets, and so we have $r_{i,j} = \binom{j-1}{i}$.

To verify the second statement, observe that as $j\to \infty$, the term in the expression corresponding to the largest $i$ dominates other terms. This allows us to uniquely recover the number of facets in ${\bf m}^{(c)}_L$. Once that has been recovered, we can subtract the corresponding term and look at the asymptotics of the remainder, and so on. 
\end{proof}

\begin{remark}\label{rem:gamma}
The \emph{$h$-vector} $(h_0,\ldots,h_{d+1})$ is a transformation of $f$-vector given by: $h(t) := \sum_{i=0}^{d+1} f_{i-1} t^i(1-t)^{d+1-i} = \sum_{i=0}^{d+1} h_i t^i.$ The polynomial $h(t)$ is called the \emph{$h$-polynomial}. Whenever a polynomial is symmetric (i.e., $t^{d+1}h(1/t) = h(t)$), it has a so-called \emph{$\gamma$-vector} given by its coefficients in terms of the basis $\{ t^i(1+t)^{d+1-2i}\}_{0\leq i\leq (d+1)/2}$. Recently many have studied complexes whose $h$-polynomials are \emph{$\gamma$-nonnegative}, as this immediately implies symmetry and unimodality of $h(t)$. Every compatibility complex that triangulates $GT_L$ has a $\gamma$-nonnegative $h$-polynomial. Indeed, by Corollary \ref{cor:fvec} it suffices to consider $h(\mnnL;t)$, which (see Remark \ref{rem:RW}) is the $h$-polynomial of the canonical triangulation of the order polytope of a certain poset $P_L$ depending only on the partial staircase $L$. By a result of Reiner and Welker \cite[Prop. 2.2]{RW}, this $h$-polynomial is nothing but the $W$-polynomial of the poset $P_L$, and a result of Br\"and\'en \cite[Thm. 4.2]{B} shows that $W(P_L;t) = h(\mnnL;t)$ is $\gamma$-nonnegative.
\end{remark}

\section{Driving rules}\label{sec:driving}

Let $T \in GT_L$ and let $T^{i,j}$ be the entry in the cell located at column $i$, row $j$ of $T$. Split the edge between cells with coordinates $(i,j)$ and $(i,j+1)$ into $T^{i,j}-T^{i,j+1}$ edges. Similarly split the edge between cells with coordinates $(i,j)$ and $(i+1,j)$ into $T^{i+1,j}-T^{i,j}$ edges. Orient all horizontal edges east and all vertical edges north. Then, for a given $T$, each node of $L$ has some incoming flow from south and west and some outgoing flow to north and east. Note that since \[(T^{i,j}-T^{i,j+1})+(T^{i+1,j}-T^{i,j}) = (T^{i+1,j+1}-T^{i,j+1})+(T^{i+1,j}-T^{i+1,j+1}),\] the total incoming flow is always equal to the total outgoing flow. 
Denote $E(T)$ the resulting configuration of edges. An example of the procedure for the GT pattern from Figure \ref{fig:dr3} can be seen on Figure \ref{fig:dr6}. 

\begin{figure}[h!]
\begin{center}
\input{drf4.pstex_t}
\end{center}
\caption{}\label{fig:dr6}
\end{figure}

Let us say that we \emph{resolve} a node of $E(T)$ if we give a bijection between the incoming and outgoing edges, manifested graphically in joining the edges by lines. By comparing Figures \ref{fig:dr5} and \ref{fig:dr6} we see that both sets of paths in Figure \ref{fig:dr5} can be obtained by resolving nodes in the appropriate manner. In fact, we have the following. 

\begin{lemma}
Any resolution of all nodes of $E(T)$ yields a decomposition of $T$ into $\{T_p\}$. Any decomposition can be obtained in this way.
\end{lemma}

\begin{proof}
Once we have resolved all nodes, we obtain a family of paths $\{p\}$ by following edges from the source according to corresponding bijections at nodes. The sum of the resulting $T_p$ does give the starting $T$ since each $T_p$ contributes $1$ to differences $T^{i,j}-T^{i,j+1}$ and $T^{i+1,j}-T^{i,j}$ exactly when it passes between the corresponding cells. 

On the other hand, if we start with a system of paths such that $T = \sum T_p$, color each path a unique color and place it on $L$. This picture clearly coincides with $E(T)$ away from the nodes, while at each node we can choose the bijection that matches colors of incoming and outgoing edges.
\end{proof}

Note that we have a certain degree of freedom when going back from a family of paths to a resolution of nodes. Namely, for a fixed node and a fixed direction, we can choose the order in which to place the edges coming from particular paths. We shall say that we choose a \emph{planar edge embedding} when we make a choice of how exactly to order the edges. 

We are now interested in resolving nodes according to certain local rules. Because of the resemblance of incoming and outgoing flows with driving flows on roads, we call these \emph{driving rules}. 

For a driving rule $r$ and an element $T \in GT_L$ we let $P^{(r)}(T)$ be the collection of paths resulting from resolving $E(T)$ according to $r$. We treat $P^{(r)}(T)$ as a multiset rather than set, and we let $\overline{P}^{(r)}(T)$ denote the corresponding set of {\it {distinct}} paths obtained after resolution. The next lemma is obvious in light of this discussion; it can almost be taken as the definition of a driving rule.

\begin{lemma}[Fundamental lemma of driving rules] \label{lem:drtri}
For any driving rule $r$, we have \[T = \sum_{p \in P^{(r)}(T)} T_p.\]
\end{lemma}

This lemma means that for any driving rule $r$ there is well-defined map $\psi^{(r)}$ from $GT_L$ to the set of subsets of $2^{\bf m}$ given by $T \mapsto$ $\{$sets corresponding to paths of $\overline{P}^{(r)}(T)\}$. We say that $r$ is a \emph{simplicial} driving rule whenever the image, $\psi^{(r)}(GT_L)$, is a simplicial complex, which happens for example if for every $T$ in $GT_L$ we have \[P^{(r)}(T) -\{p\} = P^{(r)}( T-T_p),\] for all $p$ in $P^{(r)}(T)$. We can then think of $\psi^{(r)}(GT_L)$ as a compatibility complex, though the specific conditions imposed on pairs of sets may not be easy to describe. Nonetheless, the following gives a partial answer to Question \ref{q:1}.

\begin{proposition}\label{prp:simplicial}
If $r$ is a simplicial driving rule, $\psi^{(r)}(GT_L)$ triangulates $GT_L$.
\end{proposition}

There are two particular driving rules we are going to consider in detail---the \emph{Michigan} driving rule, $mich$, and the \emph{Boston} driving rule, $bos$. Both are simplicial; indeed, we will see that $\psi^{(mich)}(GT_L) = \mnnL$ and $\psi^{(bos)}(GT_L) = \mncL$.

\begin{figure}[h!]
\begin{center}
\input{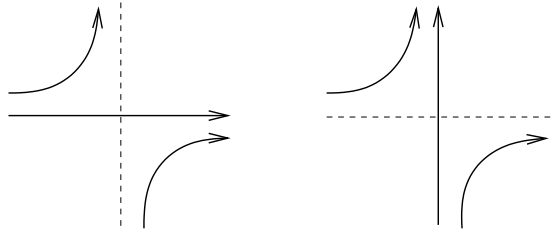}
\end{center}
\caption{Michigan driving rule.}\label{fig:dr7}
\end{figure}

The first rule can be described as follows. Let $a$ be the numbers of edges (or drivers) entering a node from the west; let $d$ be the number from the south. Similarly, let $b$ and $c$ be the number of drivers leaving this node from north and east, respectively. Then we know $a + d = b + c$. 

\begin{definition}[Michigan driving rule] 
The \emph{Michigan driving rule} resolves each node of $E(T)$ according to the following two cases:
\begin{enumerate}
    \item If $a \geq b, d \leq c$, the leftmost $b$ drivers from the west turn left without crossing routes with each other. The remaining $a-b$ drivers go straight, without crossing routes. Finally, all $d$ drivers coming from the south turn right, without crossing routes with each other or with drivers coming from the west.
    \item If $a \leq b, d \geq c$, all $a$ drivers from the west turn left without crossing routes with each other. The leftmost $b-a$ drivers from the south go straight without crossing routes with each other or any west-coming drivers. Finally, the remaining $c$ drivers coming from the south turn right without crossing routes.
\end{enumerate}
\end{definition}

The Michigan driving rule is schematically depicted on Figure \ref{fig:dr7}. The bold arrows denote the flows of drivers. Note that if $a=b, c=d$ then everybody turns and nobody goes straight: west-coming drivers turn north and south-coming drivers turn east. It is easy to remember the Michigan driving rule by recalling that drivers never cross paths.

\begin{figure}[h!]
\begin{center}
\input{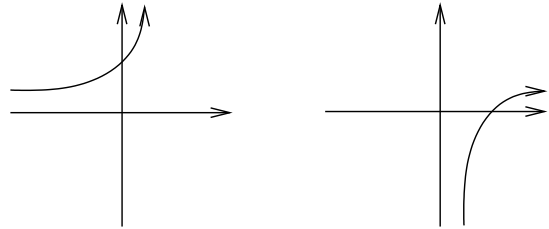}
\end{center}
\caption{Boston driving rule.}\label{fig:dr8}
\end{figure}

\begin{definition}[Boston driving rule]
The \emph{Boston driving rule} resolves each node of $E(T)$ according to the following two cases:
\begin{enumerate}
    \item If $a \geq c, d \leq b$, the leftmost $a-c$ drivers from the west turn left without crossing routes with each other. The remaining $c$ drivers go straight without crossing routes with each other. Finally, all $d$ drivers coming from south go straight, without crossing routes with each other, but crossing the routes of all other drivers. 
    \item If $a \leq c, d \geq b$, all $a$ drivers from the west go straight without crossing routes with each other. The leftmost $b$ drivers from the south go straight, without crossing routes with each other but crossing the routes of all west-coming drivers. Finally, the remaining $d-b$ drivers coming from south turn right, without crossing routes with each other but crossing the routes of all west-coming drivers.
\end{enumerate}
\end{definition}

The Boston driving rule is schematically depicted on Figure \ref{fig:dr8}. Note that if $a=c, b=d$ then everybody goes straight and nobody turns. One can remember the Boston driving rule by remembering that drivers take every opportunity possible to cross the path of perpendicular traffic (but they do turn from the correct lane). 

We can see on the left of Figure \ref{fig:dr5} that drivers follow the Michigan driving rule, while on the right drivers follow the Boston driving rule. This is no accident, and we will see in Sections \ref{sec:Michigan} and \ref{sec:Boston} how the Michigan and Boston rules can be used to study $\mnnL$ and $\mncL$, respectively.

\begin{remark}
Although Theorem \ref{thm:path} characterizes weak separability in terms of paths, we do not have a driving rule to resolve a pattern into pairwise weakly separable paths. This is because not all patterns can be so separated! However, since every collection of weakly separable paths is also a collection of non-kissing paths, it will follow from Corollary~\ref{thm:nctri} that every pattern can be resolved into a collection of weakly separated paths in at most one way. 
\end{remark}



\section{The Michigan driving rule}\label{sec:Michigan}

In this section, we will use the Michigan driving rule to establish the basis properties of $\mnnL$. These results are classical; in the next section we will establish similar results for $\mncL$, which will be new. The following theorem describes the key property of the Michigan rule. 

\begin{theorem} \label{thm:Michigannc}
If $T \in GT_L$, the paths in $P^{(mich)}(T)$ are pairwise non-crossing. Conversely, for every pairwise non-crossing collection of paths $P$, there exists a planar edge embedding so that the nodes of $E(\sum_{p \in P} T_p)$ are resolved exactly according to the Michigan rule. In other words, the Michigan driving rule is simplicial and \[ \psi^{(mich)}(GT_L) = \mnnL.\]
\end{theorem}

\begin{proof}
The first part follows from the simple observation that, with the Michigan driving rule, the routes of drivers never intersect. 

For the second part order the paths $p \in P$ lexicographically, considering paths as words on the alphabet $\{N < E\}$. In other words, $p_i < p_j$ if, at the first place where the paths diverge, $p_i$ steps north and $p_j$ steps east. Order coinciding paths arbitrarily. Now we are going to place the paths in $L$ so that the local picture at each node follows the Michigan driving rule.

In the neighborhood of each node we assign the edges of $E(\sum_{p \in P} T_p)$ to the paths in $P$ which are passing through this node in the obvious way. We order the incoming edges first from north to south among the east-going edges, then west to east among the north-going edges. Likewise order the outgoing edges first from west to east among north-going edges, then north to south among east going edges. We now embed the $i$th path passing through this node by connecting the $i$th incoming edge with the $i$th outgoing edge. This obviously obeys the Michigan driving rule.
\end{proof}

Theorem \ref{thm:Michigannc} along with Proposition \ref{prp:simplicial} gives the following corollary.

\begin{corollary}\label{thm:nntri}
The complex $\mnnL$ triangulates $GT_L$.
\end{corollary}


%
%

Recall that a simplicial complex is called a pseudo-manifold if it is pure of dimension $d$ (for some $d$) and every face of dimension $d-1$ is contained in $2$ faces of dimension $d$; it is called a pseudo-manifold with boundary if every face of dimension $d-1$ is contained in either $1$ or $2$ faces of dimension $d$.

\begin{corollary} \label{cor:purenn}
The simplicial complex $\mnnL$ is homeomorphic to a ball of dimension $N(L)$.  It is pure of dimension $N(L)$, a pseudo-manifold with boundary, and Cohen-Macaulay.
\end{corollary}

\begin{proof}
That $\mnnL$ is a ball of dimension $N(L)$ follows immediately from Theorem~\ref{thm:nntri} and Proposition~\ref{pro:trigeom}. In particular, $\mnnL$ is pure of dimension $N(L)$. A simplicial complex that is a manifold with boundary is, in particular, a pseudo-manifold with boundary. Since $\mnnL$ triangulates a ball, it is Cohen-Macaulay by~\cite[Theorem~2.2]{Munkres}. 
\end{proof}

\section{The Boston driving rule}\label{sec:Boston}

The following theorem describes the key property of the Boston rule.

\begin{theorem} \label{thm:Bostonnc}
If $T \in GT_L$, the paths in $P^{(bos)}(T)$ are pairwise non-kissing. Conversely, for every pairwise non-kissing collection of paths $P$, there exists a planar edge embedding so that the nodes of $E(\sum_{p \in P} T_p)$ are resolved exactly according to the Boston rule. In other words, the Boston rule is simplicial and \[ \psi^{(bos)}(GT_L) = \mncL.\]
\end{theorem}

\begin{proof}
The first part follows by examination of the Boston driving rule. Two paths entering a node from different directions cannot leave the node in different directions without crossing. If one path turns to follow the other for some distance before parting ways, then they must cross upon meeting and not cross when parting. In particular, they will not kiss. 

Now, let $P$ be any collection of pairwise non-kissing paths. Set $\mathcal{E}=E(\sum_{p \in P} T_p)$. Let $e$ be any edge of $L$ and let $P(e)$ be the collection of paths passing through $e$. For each path $p$ in $P(e)$, let $w_{e \to}(p,e)$ be the word on the alphabet $\{ N < E \}$, describing the path $p$ takes after $e$ and let $w_{\from e}$ be the word in the alphabet $\{ S < W \}$ describing the reverse of the path $p$ takes before $e$. Order $P(e)$ lexicographically by $w_{e \to}(p,e)$; we call this the \emph{first ordering}. Note that two different paths in $P(e)$ may have the same $w_{e \to}(\cdot,e)$ and thus be tied in the first ordering. Also, order $P(e)$ lexicographically by $w_{\from e}(p,e)$; we call this the second ordering. By the non-kissing hypothesis, these two orderings will be consistent, but some paths that are tied in one ordering will not be tied in the other. Any paths which in tied in both the first and second ordering are identical. Define the \emph{final ordering} of $P(e)$ to be the common refinement of the first ordering and the second.

If $e$ points eastward, order the copies of $e$ in $\mathcal{E}$ from north to south and assign them to the paths in $P(e)$ in the order of the final ordering. If $e$ points northward, order the 
copies of $e$ in $\mathcal{E}$ from west to east and, again, assign them to the paths in $P(e)$ in the order of the final ordering. We claim that this way of drawing the paths obeys the Boston driving 
rule.

Fix an internal node $v$. Note that, by the non-kissing hypothesis, it is impossible that there is one path that enters $v$ heading east and leaves to the north and another path that enters $v$ heading north and leaves to the east. Without loss of generality, assume that the latter case does not occur. Let $P_{NN}$, $P_{EN}$ and $P_{EE}$ be the set of paths that enter and then leave $v$ in the directions $NN$, $EN$ and $EE$ respectively. Within each of these three groups, our algorithm places the paths in the same order when they enter $v$ and when they leave it. Since the final ordering refines the first ordering, along the edges entering $v$ from the west, those that will turn to the north are placed to the north of those that will continue straight to the east. Similarly, since the final ordering refines the second ordering, among the edges leaving $v$ towards the north, those that came from the west will lie to the east of those that came from the south. In short, the paths obey the Boston driving rule at $v$.

%
\end{proof}

As with the Michigan rule, Theorem \ref{thm:Bostonnc} along with Proposition \ref{prp:simplicial} gives the following.

\begin{corollary} \label{thm:nctri}
The complex $\mncL$ triangulates $GT_L$.
\end{corollary}


As before, we obtain another corollary.

\begin{corollary} \label{cor:purenc}
The simplicial complex $\mncL$ is homeomorphic to a ball of dimension $N(L)$.  It is pure of dimension $N(L)$, a pseudo-manifold with boundary, and Cohen-Macaulay.
\end{corollary}

It is worthwhile to give a second proof that $\mncL$ has dimension $N(L)$. This argument can also be adapted to the case of $\mnnL$.

\begin{proof}[Second proof that $\mncL$ has dimension $N(L)$ (sketch)]
 In what follows it is convenient to look at the number of ``traffic flows" or ``streams" we have without distinguishing how many drivers take a certain path. Recall that we denote the set of distinct 
flows by $\overline{P}$. For a collection of paths $P$
supported on a face $F= F(P)$ of $\mncL$, the dimension of $F$ is $|\overline P|-1$.

Recall that a critical node is a node which has edges leaving it heading to the north and to the east, and that $N(L)$ is the number of critical nodes. At each critical node, the number of traffic streams increases by at most $1$. This happens exactly when one of the entering streams splits, \emph{i.e.}, when part of it turns and part does not. It is also possible that no stream splits, in which case the number of streams is preserved. Thus we can have at most $N(L)+1$ distinct streams in $P$: two streams originating from the source and one extra stream for every other critical node. So the face $F(P)$ has dimension at most $N(L)$.
\end{proof}

This proof is useful because, if we are given a collection $P$ of non-kissing paths with cardinality less than $N(L)+1$, this proof tells us how to search for paths we can add to $P$ to enhance it to a maximal collection: we need to introduce branching at nodes that don't currently have any. This is particularly nice because, at any node, there are always at most two options for which path to make branch. (Although there are many options for what route the new path should take after the branch.) In particular, this approach has proved useful when checking cases of Conjecture~\ref{conj:lzk} by hand.

One can now easily verify the upper bound of \cite{LZ} on the size of families of pairwise weakly-separated subsets. Let $2^{\mathbf{m}}_{k\leq l}$ denote the set of all subsets $I$ of $[m]$ such that $k\leq |I|\leq l$.

\begin{corollary}\cite[Theorem 1.3]{LZ}
The maximal size of a family of pairwise weakly separated subsets of $2^{\mathbf{m}}_{k\leq l}$ is \[\binom{m+1}{2} - \binom{m+1-l}{2}-\binom{k+1}{2}+1.\]
\end{corollary}

\begin{proof}
Indeed, choose $L$ to be the partial staircase shape with sinks $(m-l, l), (m-l+1,l-1), \ldots, (m-k, k)$. By Lemma \ref{lem:ncws} any weakly separated family is also non-crossing, and by Theorem \ref{cor:purenc} the size of such a family is at most $N(L)+1$. It remains to check that the number of critical nodes in $L$ is exactly $\binom{m+1}{2} - \binom{m+1-l}{2} - \binom{k+1}{2}$.
\end{proof}

Note that in this light the conjectural purity of $\mwsL$ is particularly surprising. It claims that despite the fact that weak separation is a stronger condition than non-crossing, the sizes of the facets of the corresponding complexes coincide.

\section{Regularity and its Consequences}\label{sec:regular}

In this section, we will show that $\mncL$ is a \emph{regular} triangulation of $GT^0_L$, and discuss the consequences of that result. 
These include showing that $\mncL$ is shellable and that our basis of the Pl\"ucker algebra is a basis of standard monomials, in the sense of Gr\"obner theory. 

Let $A$ be a finite subset of $\RR^N$ and let $w$ be a function from $A$ to $\RR$. Let $\Hull(A)$ be the convex hull of $A$. Then we can use $w$ to build a polyhedral subdivision of $\Hull(A)$, known as the \emph{regular subdivision induced by $w$}. The definition is as follows: Let the set $A'$ in $\RR^{N+1}$ be the set of points $\{ (a, w(a)) \}_{a \in A}$. Let $\Hull(A')$ be the convex hull of $A'$. The regular subdivision of $A$ induced by $w$ is the projection down to $\Hull(A)$ of the lower faces of $\Hull(A')$. See \cite[Chapter~5]{Zieg} for background on regular subdivisions and \cite[Chapter~8]{Sturm} for the connection between regular subdivisions and Gr\"obner theory.

The following alternative description of regular \emph{triangulations} will be more useful in our setting: Let $A$ be a finite subset of $\RR^N$, let $\Delta$ be a triangulation of $\Hull(A)$ and let $w$ be a real valued function on $A$. Assume further that $A$ is contained in an affine hyperplane which does not contain the origin. Then $\Delta$ is the regular subdivision induced by $w$ if and only if the following condition holds: If $\Hull(p_1, p_2, \cdots, p_r)$ is any face of $\Delta$ and $\sum_{i=1}^r c_i p_i=\sum_{a \in A} d_a a$ is any true relation in which all the coefficients $c_i$ and $d_a$ are positive, then $\sum c_i w(p_i) \leq \sum d_a w(a)$, with equality if and only if the two sums $\sum_{i=1}^r c_i p_i$ and $\sum_{a \in A} d_a a$ contain the same terms with the same coefficients. Also, if all of the points in $A$ are lattice points, we can restrict to the case where the $c_i$ and $d_a$ are integers.

We will now exhibit a specific weight function $w$ on the set $\{ T_p \}$, where $p$ ranges over paths in $L$, such that $\mncL$ is the regular subdivision of $GT^0_L$ induced by $w$. For this purpose, it is most convenient to describe paths by their words in $\{ N, E \}$. Let $p$ be a path contained in $L$ and let $u_1 u_2 \cdots u_m$ be the corresponding word in the alphabet $\{ N, E \}$. For $1 \leq i < j \leq m$, define $\delta_{ij}(p)$ to be $0$ if $u_i=u_j$ and $1$ otherwise. We define
$$w(T_p)=\sum_{1 \leq i < j \leq m} \delta_{ij}(p) \epsilon^{j-i}$$
where $\epsilon$ is a very small positive real number. Explicitly, we can take $\epsilon=1/\binom{m}{2}$. We want to use the criterion of the previous paragraph, so we need to know that there is an affine hyperplane which contains all of the $T_p$ but does not contain the origin; that hyperplane is given by the equation $T^{1,-1}=1$.

\begin{theorem}
The regular subdivision of $GT^0_L$ induced by $w$ is $\mncL$.
\end{theorem}

\begin{proof}
Let $Q$ be any collection of paths in $L$, set $S=\sum_{q \in Q} T_q$ and let $P=P^{(bos)}(S)$. We want to show that $\sum_{p \in P} w(T_p) \leq \sum_{q \in Q} w(T_q)$, with equality only if $P=Q$. 

There are only finitely many collections of paths $R$ with $\sum_{r \in R} T_r=S$. Let $R$ be a collection which, among all such collections, minimizes $\sum_{r \in R} w(T_r)$. We will show that $R$ must be non-kissing, and hence must equal $P$. This proves the desired result, because $\sum_{r \in R} w(T_r) \leq \sum_{q \in Q} w(T_q)$ by construction, and the equality condition follows because we will show that $P$ is the unique choice of $R$ minimizing this expression.

So, assume for the sake of contradiction that $R$ is not non-kissing. Let $p$ and $q$ be two paths in $R$ which come together after their first $a$ steps and then depart without crossing at step $b$. Let $u_1 u_2 \cdots u_m$ be the $\{ N,E \}$-word for $p$ and let $v_1 v_2 \cdots v_m$ be the $\{ N,E \}$-word for $q$. Define the paths $p'$ and $q'$ by the words $u_1 u_2\cdots u_{b-1} v_b v_{b+1} \cdots v_m$ and  $v_1 v_2 \cdots v_{b-1} u_b u_{b+1} \cdots u_m$. It is clear that $T_p+T_q=T_{p'}+T_{q'}$. For all $1 \leq  i < j \leq m$ with $j-i < b-a$, we have $\delta_{ij}(p)+\delta_{ij}(q)=\delta_{ij}(p')+\delta_{ij}(q')$. (When $j<b$, we have $\delta_{ij}(p)=\delta_{ij}(p')$ and  $\delta_{ij}(q)=\delta_{ij}(q')$; when $i > a$ the reverse holds.) We also have this equality when $j-i=b-a$, but $(i,j) \neq (a,b)$. On the other hand, $\delta_{ab}(p)=\delta_{ab}(q)=1$ while $\delta_{ab}(p')=\delta_{ab}(q')=0$. So $w(T_{p})+w(T_{q})$ contains two $\epsilon^{b-a}$ terms that $w(T_{p})+w(T_{q})$ does not, and all other differences between $w(T_{p'})+w(T_{q'})$ and $w(T_{p})+w(T_{q})$ involve higher powers of $\epsilon$. (And there are at most $\binom{m}{2}-1$ such other differences.) So, if we take $\epsilon$ small enough, then $w(T_{p})+w(T_{q}) > w(T_{p'})+w(T_{q'})$. Let $R'$ be the collection of paths obtained by taking $R$ and replacing $p$ and $q$ with $p'$ and $q'$. Then $\sum_{r \in R'} T_r=S$ and $\sum_{r \in R} w(T_r) > \sum_{r \in R'} w(T_r)$. This contradicts our choice of $R$ as minimal, and proves the theorem.
\end{proof}

\begin{corollary}
With respect to a certain term order, the monomial basis for the Pl\"ucker algebra we constructed from $\mncL$ is the basis of standard monomials. 
\end{corollary}

\begin{proof}
By~\cite[Corollary~11.6(2)]{Sturm}, there is a term order such that $\mncL$ is an ``initial complex'' of the Pl\"ucker algebra (in the language of~\cite{Sturm}). The triangulation $\mncL$ is unimodular, meaning that, if $P$ is any collection of non-kissing paths, and $x$ is in the positive real span of $\{ T_p \}_{p \in P}$, then $x$ is a positive \emph{integral} combination of $\{ T_p \}_{p \in P}$. By~\cite[Corollary~8.9]{Sturm}, when the regular triangulation is unimodular, the corresponding initial ideal is reduced, meaning that the standard monomials are precisely the monomials which are supported on faces of the initial complex.
\end{proof}

As another corollary, we deduce that $\mncL$ is shellable.

\begin{corollary}
 The simplicial complex $\mncL$ is shellable.
\end{corollary}

\begin{proof}
By \cite[Corollary~8.14]{Zieg}, regular subdivisions of polytopes are always shellable.
\end{proof}

\begin{definition}[Solid paths]
A path $p$ in $L$ is \emph{solid} if it does one of the following: 
\begin{enumerate}
  \item $p$ moves east (possibly zero  distance), then moves north until it cannot do so any longer, then it moves east (possibly zero distance) until it hits a sink,
  \item $p$ moves north (possibly zero distance), then moves east until it cannot do so any longer, then it moves north (possibly zero distance) until it hits a sink.
\end{enumerate}
\end{definition}

Similarly, call an element $I$ of $2^{\bf m}_L$ (i.e., $I$ such that $p(I)$ fits in $L$) \emph{solid} if the corresponding path $p(I)$ is solid.

\begin{example}
For example, the path shown on Figure \ref{fig:dr3} corresponding to $\{2,3,6\}$ is not solid, while the paths corresponding to $\{3,4,5\}$, $\{1,2\}$ or $\{1,2,3,4,5,6\}$ are solid. 
\end{example}


\begin{lemma} \label{lem:frozen}
Every solid set $I$ in $2^{\bf m}_L$ is non-crossing with every other element of $2^{\bf m}_L$, and thus belongs to every facet of $\mncL$.
\end{lemma}

\begin{proof}
One easily verifies that the solid paths $p$ cannot kiss any other path inside $L$, so by Theorem \ref{thm:path} the lemma follows.
\end{proof}

Let $2^{*}_L$ be a subset of $2^{\bf m}_L$ obtained by excluding the solid elements. Denote $\mncL^*$ the restriction of $\mncL$ to vertex set $2^{*}_L$. By Lemma \ref{lem:frozen} the complex $\mncL$ is a repeated cone over $\mncL^*$. Thus to understand the topology of $\mncL$ it suffices to understand the topology of $\mncL^*$.

\begin{theorem} \label{thtm:ncps}
The complex $\mncL^*$ is a topological sphere.
\end{theorem}

\begin{proof}
Let $d$ be the dimension of $\mncL^*$ (which we know is pure). First, we argue that $\mncL^*$ is a pseudomanifold without a boundary, i.e., that every face of dimension $d-1$ is contained in precisely two faces of dimension $d$. Let $F$ be a face of $\mncL^*$ of dimension $d-1$. Let $F'$ be the face of $\mncL$ which contains $F$ and all of the solid elements, so $F'$ has dimension $N(L)-1$. We claim that, under the isomorphism $\mncL \cong GT^0_L$, the face $F'$ is not contained in the boundary of $GT^0_L$. If $F'$ were contained in the boundary of $GT^0_L$ then it would be contained in some facet of $GT^0_L$. The defining equations of these facets are of the form $T^{i,j}=T^{i+1, j}$ and $T^{i, j-1}=T^{i,j}$, where $(i,j)$ and $(i+1,j)$, or $(i, j-1)$ and $(i,j)$, are adjacent cells of $L$ separated by an edge of $L$. (The corresponding defining inequalities of the polytope $GT^0_L$ are formed by replacing ``$=$'' with ``$<$''.) But , for every one of these equations, one can find a solid path $p$ such that $T_p$ does not obey it. Thus, $F'$ is not contained in the boundary of $GT^0_L$. So $F'$ must lie in two facets of $\mncL$ and, thus, $F$ lies in two facets of $\mncL^*$.

The shellability of $\mncL$ is equivalent to shellability of $\mncL^*$, and it remains to recall the well-known result \cite[Corollary 1.28]{H} that a shellable pseudomanifold without a boundary is a topological sphere.
\end{proof}

\begin{remark}
Gal has conjectured \cite[Conj. 2.1.7]{Gal} that all \emph{flag} homology spheres are $\gamma$-nonnegative, a strengthening of the well-known Charney-Davis conjecture. Every compatibility complex is clearly flag (a minimal non-face is a pair), and as described in Remark \ref{rem:gamma} they are $\gamma$-nonnegative. Thus, $\mncL^*$ satisfies Gal's conjecture.
\end{remark}

\begin{remark}
Let $V$ be the vector space spanned by the differences $T_p-T_q$, where $p$ and $q$ range over pairs of solid paths. Let $\Gamma_L$ be the image of the polytope $GT^0_L$ in the quotient by $V$. It is not hard to show that every facet of $GT^0_L$ contains all but one of the solid paths. Thus, by \cite[Proposition 2.3]{Ath}, $\mncL^*$ triangulates the boundary of $\Gamma_L$. However, it is not true that $\Gamma_L$ is a simplicial polytope; rather, some faces in the boundary of $\Gamma_L$ are subdivided in $GT^0_L$. We do not know whether there is a simplicial polytope whose boundary is isomorphic to $\mncL^*$, but we conjecture that there is.  
\end{remark}

\begin{conjecture}
For any partial staircase shape $L$, complex $\mncL^*$ is polytopal.
\end{conjecture}

\begin{example}
Let us take $m=5$ and take the shape $L$ with sinks $(1,4)$ and $(3,2)$.  
\begin{figure}[h!]
\begin{center}
\input{drf7.pstex_t}
\end{center}
\caption{}\label{fig:dr12}
\end{figure}
Then the associated complex $\mncL^*$ and a realization of the dual polytope are shown on Figure \ref{fig:dr12}. Here the outer triangle $\{(2,3),(1,4),(2,4)\}$ should also be understood as a face.

\end{example}

\begin{example}
If $L$ is a $2 \times n$ (or an $n \times 2$) rectangle, $\mncL^*$ is the type $A$ {\it {cluster complex}} of \cite{FZ}. It is known to be polytopal, and the dual polytope is known as the \emph{associahedron}. However, $\Gamma_L$ is usually not simplicial. The first counter-example is when $n=4$ (so $m=6$). Here $\Gamma_L$ has a square face whose vertices correspond to $14$, $15$, $25$ and $24$. In $\mncL^*$, this square is subdivided into two triangles, along the diagonal joining $(1,5)$ and $(2,4)$. 
\end{example}


\begin{example}
 Let $L$ be a $3 \times 3$ rectangle (so $m=6$). In this example, we will explore the difference between $\mncL$ and $\mwsL$. There are $6$ solid paths and $N(L)=9$, so $\mncL^*$ is a $3$-sphere. We 
write $\mwsL^*$ for the subcomplex of $\mncL^*$ corresponding to weakly separated paths. There are two pairs of $3$-element subsets of $[6]$ which are non-crossing but not weakly separated, namely the pairs $(145,236)$ and $(124, 356)$. (The first pair of paths crosses twice; the second pair has an hourglass.) Each of these pairs corresponds to an edge in $\mncL^*$. Each of these edges is surrounded by four tetrahedra and these tetrahedra fit together to form an octahedron subdivided around a central axis. These two octahedra are disjoint from one another. In $\mwsL^*$, these two octahedra are removed, leaving behind a complex homeomorphic to $S^2 \times [0,1]$. The endpoints of this product are a pair of $2$-spheres, each triangulated as the boundary of the octahedron. The simplicial complex $\mwsL^*$ is a subcomplex of the $D_4$-cluster complex, which is again a $3$-sphere. In the $D_4$-cluster complex, two new vertices are added. One of these vertices is compatible with the vertices is the first octahedron, and the other with the vertices in the second. Thus, the $D_4$-cluster complex caps off the two open ends of $\mwsL^*$ to form a $3$-sphere. 
\end{example}

\end{document}